\documentclass[onefignum,onetabnum]{siamart190516}

\usepackage{amsmath,amssymb}
\usepackage{amsfonts}
\usepackage{graphicx}
\usepackage{epstopdf}
\usepackage{algorithmic}
\ifpdf
  \DeclareGraphicsExtensions{.eps,.pdf,.png,.jpg}
\else
  \DeclareGraphicsExtensions{.eps}
\fi

\crefname{hypothesis}{Hypothesis}{Hypotheses}
\headers{Mean Dimension of Ridge Functions}{C. Hoyt and A. B. Owen}

\title{
Mean Dimension of Ridge Functions\thanks{Submitted to the editors DATE.
\funding{
National Science Foundation, IIS-1837931 and, DMS-1521145.
}}}

\author{
Christopher R. Hoyt\thanks{Stanford University
  (\email{crhoyt@stanford.edu}).}
\and Art B. Owen\thanks{Stanford University 
  (\email{owen@stanford.edu).}
}}

\usepackage{amsopn}
\ifpdf
\hypersetup{
  pdftitle={Asymptotic Mean Dimension},
  pdfauthor={D. Doe, P. T. Frank, and J. E. Smith}
}
\fi

\usepackage{paralist}
\renewcommand{\le}{\leqslant}
\renewcommand{\ge}{\geqslant}
\renewcommand{\emptyset}{\varnothing}

\newcommand{\real}{\mathbb{R}}
\newcommand{\bsx}{\boldsymbol{x}}
\newcommand{\bsy}{\boldsymbol{y}}
\newcommand{\bsz}{\boldsymbol{z}}
\newcommand{\rd}{\,\mathrm{d}}
\newcommand{\hk}{\mathrm{HK}}
\newcommand{\tran}{\mathsf{T}}
\newcommand{\tf}{\tilde f}
\newcommand{\e}{\mathbb{E}}
\newcommand{\var}{\mathrm{Var}}
\newcommand{\dustd}{\mathbb{U}}
\newcommand{\olt}{\overline{\tau}}
\newcommand{\dnorm}{\mathcal{N}}
\newcommand{\sm}{\setminus}
\newcommand{\err}{\varepsilon}
\newcommand{\one}{\boldsymbol{1}}

\begin{document}

\maketitle

% REQUIRED
\begin{abstract}
We consider the mean dimension of some
ridge functions of spherical Gaussian random
vectors of dimension $d$.
If the ridge function is Lipschitz continuous,
then the mean dimension remains bounded
as $d\to\infty$.
If instead, the ridge function is discontinuous,
then the mean dimension depends on a measure
of the ridge function's sparsity, and absent sparsity
the mean dimension can grow proportionally to $\sqrt{d}$.
Preintegrating a ridge function yields a new,
potentially much smoother ridge function.
We include an example where, if one of the ridge 
coefficients is bounded away
from zero as $d\to\infty$, then preintegration
can reduce the mean dimension from $O(\sqrt{d})$
to $O(1)$.
\end{abstract}

% REQUIRED

% REQUIRED

\section{Introduction}
% Remember to state contribution

Numerical integration of  high dimensional functions
is a very common and challenging problem.  
Under the right conditions, quasi-Monte Carlo (QMC) sampling
and randomized QMC (RQMC) sampling can be very effective.
A good result can be expected from (R)QMC if the following
conditions, described in more detail below, all hold:
\begin{compactenum}[\quad1)]
\item the (R)QMC points have highly uniform low dimensional projections,
\item the integrand is nearly a sum of low dimensional parts, and
\item those parts are regular enough to benefit from (R)QMC.
\end{compactenum}
The first condition is a usual property of (R)QMC points.
In a series of papers, 
Griebel, Kuo and Sloan 
\cite{grie:kuo:sloa:2010, grie:kuo:sloa:2013, grie:kuo:sloa:2017}
address the third condition
by showing that the low dimensional parts of $f$
(defined there via the ANOVA decomposition)
are at least as smooth as the original integrand and are often much
smoother. They include conditions under which lower order ANOVA terms
of functions with discontinuities (jumps) or discontinuities in
their first derivative (kinks)  are smooth.
An alternative form of regularity, instead of smoothness, 
is for the low dimensional parts to have QMC-friendly discontinuities
as described in \cite{wang:2000}.
In this article we explore sufficient conditions for the remaining second condition
to hold.  We use the mean dimension \cite{dimdist} to quantify
the extent to which low dimensional components
dominate the integrand.

This article is focused on ridge functions
defined over $\real^d$.  Ridge functions
take the form $f(\bsx) = g(\Theta^\tran\bsx)$
for an orthonormal projection matrix $\Theta\in\real^{d\times r}$
where $r\ll d$, with $r=1$ being an important special case.
Ridge functions are useful here because we can 
find their integrals via low dimensional integration
or even closed form expressions.  That lets us investigate
the impact of some qualitative features of $f$ on
the integration problem.  Additionally, many functions in
science and engineering are well approximated by
ridge functions with small values of $r$ \cite{cons:2015}, 
so good performance on ridge functions could extend well to 
many functions in the natural sciences.
As one more example, the value of a European option under geometric Brownian
motion is a ridge function of the Brownian increments
and this is what allows the formula of Black and Scholes to be applied \cite{glas:2004}.

Our main finding is that there is an
enormous difference between functions $g(\cdot)$
with jumps and functions with kinks.
This is perhaps surprising.  Based on criteria for
finite variation in the sense of Hardy and Krause,
one might have thought that a jump in $d$ dimensions
would be similar to a kink in $d-1$.
Instead, we find that for 
Lipschitz continuous $g:\real\to\real$, the mean dimension of $f$
is bounded as $d\to\infty$ and that bound can be quite low.
For $g$ with step discontinuities, we find that the mean dimension
can easily grow proportionally to $\sqrt{d}$.
These effects were seen empirically in \cite{rhalton-tr} 
where ridge functions were used to illustrate
a scrambled Halton algorithm.
Preintegration \cite{grie:kuo:leov:sloa:2018}
turns a ridge function over $[0,1]^d$
with a jump into one with a kink, and ridge functions 
of Gaussian variables containing a jump can even
become infinitely differentiable.
The resulting Lipschitz constant need not be small.  
%Preintegration induces smoothness
%but does not always lower mean dimension.
For a linear step function we find that
preintegration can either increase mean dimension
or reduce it from $O(\sqrt{d})$ to $O(1)$.

An outline of this paper is as follows.
Section~\ref{sec:back} provides notation and background
concepts related to quasi-Monte Carlo and mean dimension.
Section~\ref{sec:ridge} introduces ridge functions and
establishes upper bounds on their mean dimension
in terms of H\"older and Lipschitz conditions and some
spatially varying relaxations of those conditions.
Corollary~\ref{cor:lipbound} there shows that a
ridge function with Lipschitz constant $C$ and variance $\sigma^2$
cannot have a mean dimension larger than $rC^2/\sigma^2$
in any dimension $d\ge r\ge 1$ for any projection $\Theta\in\real^{d\times r}$.
Section~\ref{sec:jumps} considers ridge functions with jumps.
They can have mean dimension growing proportionally to $\sqrt{d}$
and sparsity of $\theta$ makes a big difference.
Section~\ref{sec:preintegration} considers the effects of preintegration
on ridge functions. The preintegrated functions are also
ridge functions with a  H\"older constant no worse than the
original function had. Preintegration can either raise or lower
mean dimension. We give an example step function where
preintegration leaves the mean dimension asymptotically proportional to $\sqrt{d}$
with an increased lead constant. In another example, preintegration
can change the mean dimension from growing proportionally to $\sqrt{d}$
to having a finite bound as $d\to\infty$.
Section~\ref{sec:numerical} computes some mean dimensions using Sobol' indices.
Section~\ref{sec:conclusions} has conclusions,
and a discussion of how generally these
results may apply.
Section~\ref{sec:appendix} is an appendix
containing the longer proofs.

\section{Background and notation}\label{sec:back}

We use $\varphi(\cdot)$ for the standard Gaussian
probability density function and $\Phi(\cdot)$ for
the corresponding cumulative distribution function.
We consider integration with respect to a $d$-dimensional
spherical Gaussian measure, 
$$
\mu \equiv \int_{\real^d} f(\bsx)(2\pi)^{-d/2}e^{-\Vert\bsx\Vert^2/2}\rd\bsx
= \int_{(0,1)^d} f(\Phi^{-1}(\bsx))\rd\bsx,
$$
where the quantile function $\Phi^{-1}(\cdot)$ is applied componentwise.
The (R)QMC approximations to $\mu$ take the form
$
%\hat\mu = \frac1n\sum_{i=1}^nf(\bsx_i)
\hat\mu = (1/n)\sum_{i=1}^n\tf(\bsx_i)
$
for points $\bsx_i\in(0,1)^d$ and $\tf(\cdot) = f\circ\Phi^{-1}(\cdot)$.
The distribution of $\bsx$  is denoted $\dnorm(0,I_d)$
or simply $\dnorm(0,I)$ if $d$ is understood from context.

\subsection*{QMC and Koksma-Hlawka}
For QMC, the Koksma-Hlawka inequality \cite{hick:2014}
\begin{align}\label{eq:khi}
|\hat\mu-\mu| \le D^*_n \times\Vert \tf \Vert_\hk
\end{align}
bounds the error in terms of the star
discrepancy $D^*_n=D^*_n(\bsx_1,\dots,\bsx_n)$ of the
points used and the total variation of $f$ in the sense
of Hardy and Krause.
Constructions with $D_n^* =O( \log(n)^{d-1}/n)$ are known
\cite{nied:1992,dick:pill:2010,sloa:joe:1994},
proving that QMC can be asymptotically better than Monte Carlo (MC)
sampling which has a root mean squared error of $O(n^{-1/2})$.
That argument requires $\Vert\tf\Vert_\hk<\infty$ which
requires at a minimum that $f$ be a bounded function on $\real^d$.
Scrambled net RQMC has a root mean squared error that is $o(n^{-1/2})$
for any $f\in L^2$ without requiring bounded variation \cite{snetvar}.

\subsection*{Kinks and jumps}
A kink function is continuous with  a discontinuity in
its first derivative along some manifold.
Griebel et al.\ \cite{grie:kuo:leov:sloa:2018}  
consider kink functions of the  
form $\max(\phi(\bsx),0)$ 
where $\phi$ is smooth. The kink takes place
within the set $\{\bsx\mid \phi(\bsx)=0\}$.
A jump function has a step discontinuity along some manifold.
Griewank et al.\ \cite{grie:kuo:leov:sloa:2018}  
consider jump functions of the 
form $\theta(\bsx)\times \max(\phi(\bsx),0)$ where $\theta$ is
also smooth.  There can be jump discontinuities
within the set $\{\bsx\mid \phi(\bsx)=0\}$.
When $\theta(\cdot)=\phi(\cdot)$, the result is a kink function.
In the rest of this paper, $\theta$ denotes a unit vector.

\subsection*{ANOVA and mean dimension}
The ANOVA decomposition applies to any measurable and square integrable function of 
$d$ independent random inputs.  In our case, those inputs will
be either $\dustd(0,1)$ or $\dnorm(0,1)$.
%For a survey, see \cite{sobomat}.
%We present it first for $f(\bsx)$ with $\bsx\sim\dustd(0,1)^d$
%then mention how it works for $\bsx\sim\dnorm(0,I)$.

We use $1{:}d$ for $\{1,2,\dots,d\}$, and for
$u\subseteq1{:}d$, we write $|u|$ for the cardinality
of $u$ and $-u$ for the complement $1{:}d\sm u$.
The point $\bsx\in\real^d$ has components $x_j$ for $j\in1{:}d$.
The point $\bsx_u\in\real^{|u|}$ has the components
$x_j$ for $j\in u$.
We abbreviate $\bsx_{-\{j\}}$ to $\bsx_{-j}$.
For $u\subseteq1{:}d$ and points $\bsx,\bsz\in\real^d$,
the hybrid point $\bsy=\bsx_u{:}\bsz_{-u}$
has $y_j=x_j$ for $j\in u$ and $y_j=z_j$ otherwise.

%When $\bsx\sim\dnorm(0,I)$ then $\bsx_u\sim\dnorm(0,I_{|u|})$.
%It has probability density function 
%$$\varphi(\bsx_u) = (2\pi)^{-|u|/2}\exp( -\Vert\bsx_u\Vert^2/2).$$

The ANOVA decomposition \cite{hoef:1948,sobo:1969,efro:stei:1981}
of $f:[0,1]^d\to\real$ is
$f(\bsx) = \sum_{u\subseteq 1{:}d} f_u(\bsx)$
where $f_u$ depends on $\bsx$
only through $\bsx_u$.
For these functions, the line integral
$\e( f_u(\bsx)\mid \bsx_{-j})=0$ whenever $j\in u$
and from that it follows that $\e( f_u(\bsx)f_v(\bsx))=0$
when $u\ne v$ and then
$$
\sigma^2 = \sigma^2(f) =\e( (f(\bsx)-\mu)^2) 
= \sum_{u:|u|>0}\sigma^2_u
$$
for variance components $\sigma^2_u=\sigma^2_u(f) = \e(f_u(\bsx)^2)$
for $u\ne0$ and $\sigma^2_\emptyset=0$.

The mean dimension of $f$ 
(in the superposition sense) is
$$
\nu(f) = \frac{ \sum_{u\subseteq1{:}d}|u|\sigma^2_u}{ \sum_{u\subseteq1{:}d}\sigma^2_u}. 
$$
If we choose $u\subseteq1{:}d$ with probability proportional to $\sigma^2_u$
then $\nu(f)$ is the average of $|u|$.  Effective dimension is commonly
defined via a high quantile of that distribution such as the 99'th percentile \cite{cafl:moro:owen:1997}.
Such an effective dimension could well be larger than the mean dimension but it is  more
difficult to ascertain.

The mean dimension and a few other quantities that we use are not well defined 
when $\sigma^2=0$.
In such cases, $f$ is constant almost everywhere and we will not ordinarily be interested
in integrating it.  We assume below, without necessarily stating it every time, that $\sigma^2>0$.

Sobol' indices are used to quantify the importance of a variable or
more generally a subset of them. We will use
the (unnormalized) Sobol' total index for variable $j$,
$$
\olt^2_j = \sum_{u:j\in u}\sigma_u^2 .
$$
More generally,  for $u\subset 1{:}d$,
we set $\olt^2_u = \sum_{v:v\cap u\ne\emptyset}\sigma^2_v$.
An easy identity from \cite{meandim}  gives
%$\nu(f) = \frac1{\sigma^2}\sum_{j=1}^d\olt^2_j. $
$\nu(f) = (1/\sigma^2)\sum_{j=1}^d\olt^2_j$. 
Sobol' \cite{sobo:1993} shows that 
$$
\olt_j^2 = \frac12\e\bigl(\bigl(f(\bsx_{-j}{:}x_j)-f(\bsx_{-j}{:}z_j)\bigr)^2\,\bigr)
$$
when $\bsx$ and $\bsz$ are independent random vectors with the same product distribution
on $\real^d$.
As a result we find that
\begin{align}\label{eq:meandimformula}
\nu(f) = \frac1{2\sigma^2}
\e\Biggl(\, \sum_{j=1}^d 
\bigl( f(\bsx)-f(\bsx_{-j}{:}z_j)\bigr)^2\Biggr).
\end{align}
The expectation in the numerator of $\nu(f)$  is a $2d$-dimensional integral
over independent $\bsx$ and $\bsz$.
It is commonly evaluated by (R)QMC.

\subsection*{Low effective dimension}
Applying~\eqref{eq:khi} componentwise yields
$$
|\hat\mu-\mu| \le \sum_u D_n^*(\bsx_{1,u},\dots,\bsx_{n,u})
\times \Vert \tf_u\Vert_\hk.
$$
The coordinate discrepancies $D_n^*(\bsx_{1,u},\dots,\bsx_{n,u})$
are known to decay rapidly when $|u|$ is small \cite{dick:pill:2010}.
If also $\Vert \tf_u\Vert_\hk$ is negligible when $|u|$ is not small
then $\tf$ can be considered to have low effective dimension
and an apparent $O(n^{-1})$ error for QMC can be observed.
Some other ways to decompose a function into
a sum of $2^d$ functions, one for each subset of $1{:}d$,
are described in \cite{kuo:sloa:wasi:wozn:2010}.
For a survey of effective dimension methods in information
based complexity, see \cite{wasi:2019}.

To avoid the dependence on finite variation and to control the 
logarithmic terms we will use a version of RQMC known 
as scrambled nets. 
Under scrambled net sampling \cite{rtms}
each $\bsx_i\sim\dustd(0,1)^d$, while collectively
$\bsx_1,\dots,\bsx_n$ remain digital nets
with probability one, retaining their low discrepancy.	
The mean squared error of scrambled net sampling decomposes as
\begin{align}\label{eq:snetdecomp}
\e((\hat\mu-\mu)^2)
=\sum_{|u|>0} \e\Bigl(\Bigl( \frac1n\sum_{i=1}^n\tf_u(\bsx_i)\Bigr)^2\Bigr)
=\sum_{|u|>0} \var\Bigl( \frac1n\sum_{i=1}^n\tf_u(\bsx_i)\Bigr)
\end{align}
where expectation refers to randomness in the $\bsx_i$ \cite{snetvar}.
If $\tf\in L^2$, then 
\begin{align}\label{eq:vslogs}
\var\Bigl( \frac1n\sum_{i=1}^n\tf_u(\bsx_i)\Bigr)  
=o\Bigl(\frac1n\Bigr)\quad\text{and}\quad  
\var\Bigl( \frac1n\sum_{i=1}^n\tf_u(\bsx_i)\Bigr)\le \Gamma\frac{\sigma^2_u}n,
\end{align} 
for some gain coefficient $\Gamma<\infty$ \cite{snxs}.
If also $\partial^u \tf_u\in L^2$ then
\begin{align}\label{eq:cuberate}
\var\Bigl( \frac1n\sum_{i=1}^n\tf_u(\bsx_i)\Bigr) 
=O\Bigl( \frac{\log(n)^{|u|-1}}{n^3}\Bigr).
\end{align}
If large $|u|$ have negligible $\sigma^2_u$ and
small  $|u|$ are smooth enough for~\eqref{eq:cuberate} to hold
then RQMC may attain nearly $O(n^{-3/2})$ root mean squared error.
The logarithmic factors in~\eqref{eq:cuberate} cannot make
the variance much larger than the MC rate because the 
bound in~\eqref{eq:vslogs} applies for finite $n$.

The ANOVA decomposition of $f$ on $\bsx\in\real^d$
is essentially the same as that of $\tf$ on $(0,1)^d$.
Specifically, $f_u(\bsx) = \tf_u(\Phi^{-1}(\bsx))$.

Discontinuities can lead to severe deterioration in the asymptotic 
behavior of RQMC.  He and Wang \cite{he:wang:2015}
obtain MSE rates of $O( n^{-1-1/(2d-1)}(\log n)^{2d/(2d-1)})$
for jump discontinuities of the form $f(\bsx)=g(\bsx)1\{\bsx\in\Omega\}$
where the set $\Omega$ has a boundary with $(d-1)$-dimensional Minkowski
content.  When $\Omega$ is the Cartesian product of a hyper-rectangle
and a $d'$-dimensional set with a boundary of $(d'-1)$-dimensional
Minkowski content, then $d'$ takes the place of $d$ in their rate.
The smaller $d'$ is, the more `QMC-friendly' the discontinuity is.

\section{Ridge functions}\label{sec:ridge}

We let $\bsx\sim\dnorm(0,I)$, choose an orthonormal matrix
$\Theta\in\real^{d\times r}$,
and define the ridge function
\begin{align}\label{eq:ridgefunr}
f(\bsx) = g\bigl( \Theta^\tran\bsx\bigr), 
\end{align}
where $g:\real^r\to\real$.
We must always have $d\ge r$ because otherwise
$\Theta^\tran\Theta=I_r$ is impossible to attain.
Our main interest is in $r\ll d$.
Ridge functions can also be defined for $\bsx\sim\dustd[0,1]^d$
but then the domain of $g$ becomes a complicated polyhedron
called a zonotope \cite{cons:2015}.

When $r=1$, we write
\begin{align}\label{eq:ridgefun1}
f(\bsx) = g\bigl( \theta^\tran\bsx\bigr), 
\end{align}
where $g:\real\to\real$.
Then, because $\theta^\tran\bsx\sim\dnorm(0,1)$ we find that
$$
\mu = \int_{-\infty}^\infty g(z)\varphi(z)\rd z
\quad\text{and}\quad
\sigma^2 = \int_{-\infty}^\infty (g(z)-\mu)^2\varphi(z)\rd z.
$$
We can get the answer $\mu$ and the corresponding RMSE $\sigma/\sqrt{n}$
under MC by one dimensional integration. For some $g$, one or
both of these quantities are available in closed form.
Note that $\mu$ and $\sigma^2$ above are both independent
of $\theta$ and even of $d$.
For more general $r\ge1$ we find that $\mu$
and $\sigma^2$ are $r$-dimensional integrals
that do not depend on $\Theta$ or on $d\ge r$.
Apart from a few remarks, we focus mostly on the case with $r=1$.

By symmetry we can take all $\theta_j\ge0$.
It is reasonable to expect that sparse vectors $\theta$ will
make the problem of intrinsically lower dimension.
Sparsity is typically defined via small
values of $\sum_{j=1}1_{\theta_j\ne 0}$.
It is common to use instead a  proxy
measure $\Vert \theta\Vert_1$, with smaller values representing
greater sparsity, relaxing an $L_0$ quantity to an $L_1$ quantity.
By this measure, the `least sparse' unit vectors
are of the form $\theta_j=\pm1/\sqrt{d}$
while sparsest are of the form $\pm e_j$ where
$e_j$ is the $j$'th standard Euclidean basis vector. %, for $j\in1{:}d$.

We will need some fractional absolute moments of the $\dnorm(0,1)$ distribution.
For $\eta>-1$ define
\begin{align}\label{eq:meta}
M_\eta = \int_{-\infty}^\infty |y|^\eta\varphi(y)\rd y
= \frac{2^{\eta/2}}{\sqrt{\pi}} \Gamma\Bigl( \frac{\eta+1}2\Bigr).
\end{align}
This is from formula (18) in an unpublished report of Winkelbauer \cite{wink:2012}.
It can be verified directly by change of variable to $x=y^2/2$.

\begin{theorem}\label{thm:holder}
Let $f$ be a ridge function described by~\eqref{eq:ridgefunr}
for $1\le r\le d$, where $g:\real^r\to\real$ satisfies a H\"older condition 
$|g(\bsy)-g(\bsy')|\le C\Vert\bsy-\bsy'\Vert^\alpha$
for $C<\infty$, $0<\alpha\le 1$, and $\bsy,\bsy'\in\real^r$.
Then the mean dimension of $f$ satisfies
\begin{align}\label{eq:holderupperbound}
\nu(f) \le 
\Bigl(\frac{C}\sigma\Bigr)^2
2^{\alpha-1}M_{2\alpha}\times\sum_{j=1}^d
\biggl(\,
\sum_{k=1}^r\Theta_{jk}^2\biggr)^\alpha,
\end{align}
where $\sigma^2 = \var(f(\bsx))$ does not depend on $d$.
\end{theorem}
\begin{proof}
Let $\bsx$ and $\bsz$  be independent $\dnorm(0,I_d)$ random vectors. 
For $j\in1{:}d$, let $\Theta_{j\cdot}$ be the $j$'th row of $\Theta$
as a row vector.
Then
$\Theta^\tran\bsx_{-j}{:}z_j-
\Theta^\tran\bsx 
%= \begin{pmatrix}
%\Theta_{j1}(z_j-x_j) & \cdots & \Theta_{jr}(z_j-x_j)
%\end{pmatrix}^\tran
=\Theta_{j\cdot}^\tran(z_j-x_j)$.
Next
\begin{align}\label{eq:holderbound}
\olt^2_j
& =\frac12\e\Bigl( \bigl(g(\bsx)-g(\bsx_{-j}{:}z_j)\bigr)^2\Bigr)
 \le\frac{C^2}2 \e\Bigl(\Vert\Theta_{j\cdot}^\tran
(z_j-x_j)\Vert^{2\alpha}\Bigr)
=
2^{\alpha-1}C^2
\Vert\Theta_{j\cdot}^\tran\Vert^{2\alpha}
M_{2\alpha}
\end{align}
because $(z_j-x_j)/\sqrt{2}\sim\dnorm(0,1)$.
Summing over $j$ gives~\eqref{eq:holderbound}.
Finally, $\sigma^2$ depends on the distribution of $g(\bsy)$
for $\bsy\sim\dnorm(0,I_r)$ which is independent of $d$.
\end{proof}

If $\alpha\ge1/2$, then we recognize
$\sum_{j=1}^d\bigl(\sum_{k=1}^r\Theta_{jk}^2\bigr)^\alpha$
as $\Vert\Theta^\tran\Vert_{2,2\alpha}^{2\alpha}$
where $\Vert\cdot\Vert_{p,q}$ is a matrix $L_{p,q}$ norm 
\cite{ostr:1955}.
For $\alpha<1/2$, we get $q<1$ and this is then not a norm.
If $Q\in\real^{d\times d}$ is an orthogonal matrix, then 
$g(\Theta^\tran\bsx) = g( (Q\Theta)^\tran (Q\bsx))$. 
Now $Q\bsx\sim\dnorm(0,I)$ so we can replace
$\Vert \Theta^\tran \Vert^{2\alpha}_{2,2\alpha}$ 
in~\eqref{eq:holderupperbound} by
$\inf_Q
\Vert \Theta^\tran Q^\tran\Vert^{2\alpha}_{2,2\alpha}$.
For $\alpha=1$, we get $\Vert\Theta^\tran\Vert^{2\alpha}_{2,2\alpha}=
\sum_{j=1}^d\sum_{k=1}^r\Theta_{jk}^2=\Vert \Theta\Vert_F^2$, the
squared Frobenius norm of $\Theta$, and the bound in~\eqref{eq:holderupperbound}
simplifies to reveal a proportional dependence on $r$.

\begin{corollary}\label{cor:lipbound}
Let $f$ be a ridge function described by \eqref{eq:ridgefunr}
where $g$ is Lipschitz continuous with constant $C$
and  $\Theta\in\real^{d\times r}$ with $\Theta^\tran\Theta=I_r$,
for $r\le d<\infty$. 
Then 
$$\nu(f) \le r\times\Bigl(\frac{C}\sigma\Bigr)^2$$
where $\sigma^2 = \var(f(\bsx))$ does not depend on $d$.
\end{corollary}
\begin{proof}
Take $\alpha=1$ in Theorem~\ref{thm:holder}.
\end{proof}

The bound in Theorem~\ref{thm:holder} and its corollaries is conservative.
It allows for the possibility that 
$|g(\bsy)-g(\bsy')|=C\Vert\bsy-\bsy'\Vert^\alpha$ for all pairs
of points $\bsy,\bsy'\in\real^r$.
If that would hold for $r=1$ and $\alpha=1$,
then it would imply that $g$ is linear.
To see why, note that any triangle with points
$(y_1,g(y_1))$,
$(y_2,g(y_2))$, and
$(y_3,g(y_3))$,
for distinct $y_j$ would have one angle equal to $\pi$.
A linear function would then have mean dimension $1$,
the smallest possible value when $\sigma^2>0$.
A less conservative bound is in Section~\ref{sec:spatiallyvarying} below.
The next result show that the bound has a dimensional effect when $\alpha<1$.

\begin{corollary}
Let $f$ be a ridge function given by \eqref{eq:ridgefun1}
with $r=1$, where $g$ is H\"older continuous with constant $C$ and exponent $\alpha\in(0,1)$
and  $\theta\in\real^d$ is a unit vector 
for $1\le d<\infty$. 
Then 
$$
\nu(f) \le \Bigl(\frac{C}\sigma\Bigr)^2
2^{\alpha-1}M_{2\alpha} d^{1-\alpha}.$$
\end{corollary}
\begin{proof}
From Theorem~\ref{thm:holder},
$\nu(f) \le 2^{\alpha-1}M_{2\alpha}(C/\sigma)^{2}\sum_{j=1}^d|\theta_j|^{2\alpha}$.
The largest value this can take arises for $\theta_j=\pm1/\sqrt{d}$.
Then 
$\sum_{j=1}^d|\theta_j|^{2\alpha} = d\times d^{-2\alpha/2}=d^{1-\alpha}$ and so
$\nu(f) \le 
2^{\alpha-1}M_{2\alpha}C^2\sigma^{-2} d^{1-\alpha}
$
as required.
\end{proof}

\subsection{Spatially varying H\"older and Lipschitz constants}
\label{sec:spatiallyvarying}

A Lipschitz or H\"older inequality provides a bound  on
$|g(\bsy)-f(\bsy')|$ 
that holds for all $\bsy,\bsy'\in\real^r$.
The numerator in $\nu(f)$ is a weighted average 
of $|f(\bsx)-f(\bsx_{-j}{:}z_j)|^2$ over points $\bsx,\bsz$ and 
indices $j$, and for a ridge function that reduces to
a weighted average of $|g(\bsy)-g(\bsy')|^2$ 
Applying a Lipschitz or H\"older inequality bounds 
an $L_2$ quantity 
by the square of an $L_\infty$ quantity.

We say that $g$ satisfies a spatially varying H\"older condition if
for some $0<\alpha\le 1$ there is a function $C(\bsy)$ such that
\begin{align}\label{eq:spaceholder}
|g(\bsy)-g(\bsy')| \le C(\bsy)\Vert\bsy-\bsy'\Vert^\alpha
\end{align}
holds for all $\bsy$ and $\bsy'$.
If $\alpha=1$, then $g$
satisfies a spatially varying Lipschitz condition.
The well known locally Lipschitz condition is different.
It requires that every $\bsy$ be within a neighborhood $U_{\bsy}$
on which $g$ has a finite Lipschitz constant $C(\bsy)$.
Equation~\eqref{eq:spaceholder} is stronger because it 
also bounds $|g(\bsy)-g(\bsy')|$ for $\bsy'\not\in U_{\bsy}$.

We will use a H\"older inequality via $1<  p\le \infty$
and $q$ satisfying $1/p+1/q=1$ to slightly modify the proof in
Theorem~\ref{thm:holder}.
Under~\eqref{eq:spaceholder}
\begin{align}
\sigma^2\nu(f) 
& \le\frac12\sum_{j=1}^d
\e\bigl( C(\Theta^\tran\bsx)^2\Vert \Theta_{j\cdot}^\tran(z_j-x_j)\Vert^{2\alpha}\bigr)\notag\\
& \le\frac12\e( |C(\bsy)|^{2p})^{1/p}
\sum_{j=1}^d\e\bigl( 
\Vert \Theta_{j\cdot}^\tran(z_j-x_j)\Vert^{2\alpha q}\bigr)^{1/q}\quad\text{(with $\bsy\sim\dnorm(0,I_r)$)}\notag\\
& \le
2^{\alpha-1}\e( |C(\bsy)|^{2p})^{1/p}
M_{2\alpha q}^{1/q}
\sum_{j=1}^d
\Vert \Theta_{j\cdot}^\tran\Vert^{2\alpha }.\label{eq:spaceholderbound}
\end{align}
Allowing $p=1$ would have made $q=\infty$ and then 
the supremum norm
of $|x_j-z_j|$ would be infinite, leading to a useless bound.  
For $p=\infty$, we interpret
$\e( |C(\bsy)|^{2p})^{1/p}$ as $\sup_{\bsy} |C(\bsy)|^2$
recovering Theorem~\ref{thm:holder}.
The bound \eqref{eq:spaceholderbound}
simplifies for $r=1$ and for $\alpha=1$. Under both simplifications,
$$
\nu(f) \le \frac1{\sigma^2} \e( C(\bsy)^{2p})^{1/p}M_{2q}^{1/q}.
$$
To get a finite bound for $\nu(f)$ 
it suffices for $C(\bsy)$
to have a finite moment of order $2+\epsilon$
for some $\epsilon>0$.

\subsection{A kink function}\label{sec:kinkfunctionbound}
As a prototypical kink function, consider $f$ given
by~\eqref{eq:ridgefun1} with $g(y) = (y-t)_+$
for some threshold $t$.
This $g$ is Lipschitz continuous
with $C=1$. 
Using indefinite integrals
$\int x\varphi(x)\rd x = -\varphi(x)+c$
and
$\int x^2\varphi(x)\rd x = \Phi(x) - x\varphi(x)+c$,
the first two moments of $f(\bsx)$ are
\begin{align*}
\mu(t) &= \int_{-\infty}^\infty \max(y-t,0)\varphi(y)\rd y = \varphi(t)-t\Phi(-t),\quad\text{and}\\
(\mu^2+\sigma^2)(t)
&=
\int_{-\infty}^\infty \max(y-t,0)^2\varphi(y)\rd y 
= \Phi(-t)(1+t^2) -t\varphi(t),\quad\text{so}\\
\sigma^2(t)&=
\Phi(-t)(1+t^2) -t\varphi(t)-\varphi(t)^2 +2t\varphi(t)\Phi(-t)-t^2\Phi(-t)^2.
\end{align*}
Because $C=M_2=1$, we get $\nu(f)\le1/\sigma^2(t)$.
For $t=0$, 
we get $\mu=\varphi(0)$ and 
$\sigma^2=\e( g(y)^2)-\mu^2=1/2-1/(2\pi)$ and
then
$$
\nu(f) \le \frac1{1/2-1/(2\pi)} = \frac{2\pi}{\pi-1} \doteq 2.933,
$$
for any $d\ge1$ and any unit vector $\theta\in\real^d$.

\subsection{The least sparse case}
The least sparse unit vectors have all $\theta_j=\pm 1/\sqrt{d}$.
Because $\dnorm(0,I)$ is symmetric we may take
 $\theta_j=1/\sqrt{d}$. 
In this case, it is easy to compute $\nu(f)$
using Sobol' indices.  By symmetry,
$\nu(f)$ equals a three dimensional integral
\begin{align}\label{eq:nufleastsparse}
\frac{d}{2\sigma^2}
%\int_{-\infty}^\infty \int_{-\infty}^\infty \int_{-\infty}^\infty
\int_{\real^3}
\biggl( g\Bigl(\frac{\sqrt{d-1}x+y}{\sqrt{d}}\Bigr)
       -g\Bigl(\frac{\sqrt{d-1}x+z}{\sqrt{d}}\Bigr)\biggr)^2
  \varphi( x) \varphi(y) \varphi(z)
\rd x\rd y \rd z,
\end{align}
for any $d\ge1$.
Furthermore, by comparing results for $d'\ll d$ to those
for $d$ we can see some impact from sparsity because
the least sparse unit vector for dimension $d'$ will give the same answer
as a very sparse $d$ dimensional vector with $d-d'$ zeros
and the remaining  components equal.

\section{Jumps}\label{sec:jumps}

While both kinks and jumps may have smooth
low dimensional ANOVA components, jumps
do not necessarily have the same low mean dimension.
They are also sensitive to sparsity of $\theta$.

\subsection{Linear step functions}
First we consider a step function
$1\{\theta^\tran\bsx>t\}$.
We get upper and lower bounds for the mean dimension
of this function in terms of the nominal dimension $d$
and $\Vert\theta\Vert_1$, our sparsity measure.
Over the range from sparsest to least sparse
$1\le\Vert\theta\Vert_1\le \sqrt{d}$.

\begin{theorem}\label{thm:stepupperbound}
Let $f(\bsx) = 1\{\theta^\tran\bsx>t\}$ for a threshold $t\ge0$
and a unit vector $\theta\in\real^d$. 
Then, for $d\ge2$, 
$$
\nu(f) \le 
 \frac{\Vert\theta\Vert_1}{\Phi(t)\Phi(-t)\sqrt{2\pi}}
 \Bigl( \sqrt{2} +2\sqrt{\log(\Vert\theta\Vert_1^{-1}d)}\Bigr)= O\bigl( \sqrt{d\log(d)}\bigr). 
$$
\end{theorem}
\begin{proof}
See Section~\ref{sec:thm:stepupperbound} of the Appendix.
\end{proof}

The $O(\sqrt{d\log(d)})$ rate
in Theorem~\ref{thm:stepupperbound}
arises for $\Vert\theta\Vert_1=\sqrt{d}$.
More generally we get 
$$O(\Vert\theta\Vert_1\sqrt{\log(d/\Vert\theta\Vert_1)})
=O\bigl(\Vert\theta\Vert_1\sqrt{\log(d)}\bigr).$$
For instance, if $\theta$ has $r\ge1$ components 
equal to $\pm 1/\sqrt{r}$
and the rest equal to zero, then the upper bound is
$O( r^{1/2}\sqrt{\log(d/r)})$.  There can thus be a significant
improvement due to sparsity of $\theta$.

\begin{theorem}\label{thm:steplowerbound}
Let $f(\bsx) = 1\{\theta^\tran\bsx>t\}$ for a threshold $t\ge0$
and a unit vector $\theta\in\real^d$. 
Then, for $d\ge2$, 
$$
\nu(f) \ge 
 \frac{\Vert\theta\Vert_1}{\Phi(t)\Phi(-t) 
2^{3/2}\pi}e^{-t^2-1}. 
$$
\end{theorem}
\begin{proof}
See Section~\ref{sec:thm:steplowerbound} of the Appendix.
\end{proof}

The proof of Theorem~\ref{thm:steplowerbound} 
requires a certain lower bound on a
bivariate Gaussian probability.  We did not find many
such lower bounds in the literature,
so this may be new and may be of independent interest.
\begin{lemma}\label{lem:bivgauslowerbound}
Let 
$$
\begin{pmatrix}
x \\ y 
\end{pmatrix}
\sim\dnorm\left(
\begin{pmatrix}
0 \\ 0 
\end{pmatrix}
,
\begin{pmatrix}
1 &\rho\\ \rho &1 
\end{pmatrix}
\right) 
$$
with $\rho\ge0$ and choose $t\ge0$. Then 
$$
\Pr( x>t, y<t ) \ge\frac1{2\pi} \Bigl(\frac{1-\rho}{1+\rho}\Bigr)^{1/2}\exp\Bigl(-\frac{t^2}{1+\rho}-1\Bigr). 
$$
\end{lemma}
\begin{proof}
See Section~\ref{sec:lem:bivgauslowerbound}
of the Appendix.
\end{proof}

Choosing $\theta =(\pm1,\pm1,\dots,\pm1)/\sqrt{d}$
in Theorem~\ref{thm:steplowerbound} 
provides an example of a set of jump functions with mean
dimension bounded below by a positive multiple of $\sqrt{d}$.
Here again sparsity plays a role in the bound.

The bounds in both Theorems~\ref{thm:stepupperbound}
and~\ref{thm:steplowerbound} depend on $t$.
The upper bound argument in
Theorem~\ref{thm:stepupperbound}
uses a mean value approximation where $\varphi(0)$
could be replaced by a value just over $\varphi(-t)$,
yielding for $t>0$ that
$$
\nu(f)\le
\frac{2\sqrt{\log(d/\Vert\theta\Vert_1)}}{\Phi(t)\Phi(-t)}
\Bigl(o(1) +\varphi(t)(1+o(1))\Bigr)
=
O\biggl(
\frac{2\sqrt{\log(d/\Vert\theta\Vert_1)}}{\Phi(t)}
\frac{t^2+1}t
\biggr)
$$
by a Mills' ratio inequality as $d\to\infty$. As a result
the upper bound is not as sensitive to large $t$ as the
presence of $\Phi(-t)$ in the denominator from Theorem~\ref{thm:stepupperbound}
would suggest.

The case $t=0$ is simpler.
We find
\begin{align*}
\nu(f) &= \frac1{\Phi(0)^2}\sum_{j=1}^d\Pr\bigl( \theta^\tran\bsx>0, \,
                  \theta^\tran\bsx+\theta_j(z_j-x_j)<0\bigr)\\
&=4\sum_{j=1}^d\int_0^\infty\varphi(x)\Phi\Bigl( -\rho_jx/\sqrt{\smash[b]{1-\rho_j^2}}\,\Bigr)\rd\bsx,\quad\rho_j=1-\theta_j^2\\
&=\sum_{j=1}^d\frac2{\pi}\Bigl(\frac\pi2-\arctan\Bigl(-\rho_j/{\sqrt{\smash[b]{1-\rho_j^2}}}\,\Bigr)\Bigr),
\end{align*}
using a definite integral from Section 2.5.2 of \cite{pate:read:1996}.
After some algebra
\begin{align}\label{eq:meandimobliquejump}
\nu(f) = \frac2\pi\sum_{j=1}^d\arcsin(|\theta_j|)
\ge\frac2\pi\Vert\theta\Vert_1.
\end{align}
Now $\arcsin(x) = x+O(x^3)$ as $|x|\to0$. Therefore
$\nu(f)\to 2\Vert\theta\Vert_1/\pi$ holds if
$\Vert\theta\Vert_\infty\to 0$ holds as $d\to\infty$.
Thus there is no asymptotic $\sqrt{\log(d)}$ factor when $t=0$
and, from the details of our proof,  we suspect it is not present
for other $t$.

\subsection{More general indicator functions}
It is reasonable to expect indicator functions
to have such large mean dimension for more general sets than just half spaces in $\real^d$
under a spherical Gaussian distribution. Here we sketch a generalization.
First, for an indicator function $f(\bsx) = 1\{\bsx\in\Omega\}$
of a measurable set $\Omega\subset\real^d$ we have
\begin{align}\label{eq:meandimindicator}
\nu(f) = \sum_{j=1}^d\e\bigl(\Pr( \bsx\in\Omega\mid\bsx_{-j})\Pr( \bsx\in\Omega^c\mid\bsx_{-j})\bigr)
\bigm/\bigl[\mu(1-\mu)\bigr]
\end{align}
for $\mu = \Pr(\bsx\in\Omega)$.
The numerator expectations are with respect to random $\bsx_{-j}$,
and~\eqref{eq:meandimindicator}
holds for any distribution on $\bsx$ with independent components, including $\dustd(0,1)^d$
and $\dnorm(0,I)$. We work with the latter case in what follows.
%The denominator in~\eqref{eq:meandimindicator} is at least $1/4$ so it suffices to
%consider what might make the numerator large.

As in \cite{grie:kuo:sloa:2013, grie:kuo:sloa:2017}
we take $\Omega = \{\bsx\mid\phi(\bsx)\ge0\}$ and place conditions on $\phi$.
Let $\phi\in C^\infty(\real^d)$ be strictly monotone in each coordinate $x_j$.
Without loss of generality, suppose that $\phi$ is strictly increasing
in each $x_j\sim\dnorm(0,1)$.  
Suppose additionally that $\lim_{z_j\to\infty} \phi(\bsx_{-j}{:}z_j)>0$
and $\lim_{z_j\to-\infty} \phi(\bsx_{-j}{:}z_j)<0$ for all $j$ and all $\bsx_{-j}\in\real^{d-1}$.
%We also assume that $\phi(0)<\infty$ so that $0\not\in\Omega$.

For any $\bsx_{-j}$, there is a unique value $z_j\in\real$ for which $\phi(\bsx_{-j}{:}z_j)=0$.
We write $\bsz^* = \bsx_{-j}{:}z_j$ and sometimes suppress its dependence on $\bsx_{-j}$.
We can make a linear approximation to the boundary of $\Omega$ at $\bsz^*$
via $\bsx^\tran\theta^*=t^*$ where both $\theta^*$, the normalized gradient of $\phi$, and $t^*$
depend on $\bsz^*$. By monotonicity of $\phi$, each $\theta^*_j>0$.
Let
\begin{align*}
\delta_j(\bsx_{-j})
&\equiv \Pr( \phi(\bsx)\ge0 \mid\bsx_{-j})\Pr( \phi(\bsx)<0\mid\bsx_{-j})\\
&= 
\Phi\Biggl(\frac{\sum_{\ell\ne j}x_\ell\theta^*_\ell(\bsx_{-j})-t^*(\bsx_{-j})}{\theta^*_j(\bsx_{-j})}\Biggr)
\Phi\Biggl(\frac{t^*(\bsx_{-j})-\sum_{\ell\ne j}x_\ell\theta^*_\ell(\bsx_{-j})}{\theta^*_j(\bsx_{-j})}\Biggr)
,\quad\text{and}\\
%\bar\delta(\bsx) & = \frac1d\sum_{j=1}^d\delta_j(\bsx_{-j}),
\delta(\bsx) & = 
\sum_{j=1}^d\delta_j(\bsx_{-j}).
\end{align*}
Now 
$\nu(f) = \e( \delta(\bsx))/[\mu(1-\mu)]$.
In words, $\e(\delta(\bsx))$ is what we would get by sampling $\bsx\sim\dnorm(0,I)$,
finding the $d$ boundary points $\bsz^*$ corresponding to the $d$ component
directions $x_j$,  summing the corresponding $\delta_j$ values, and averaging the results
over all samples.
Each point $\bsx$ leads to consideration of $d$ points $\bsz^*\in\partial\Omega$.
This process produces an unequally
weighted average over points $\bsz^*\in\partial\Omega=\{\bsz\mid \phi(\bsz)=0\}$
of a sum of $\delta_j$ values determined by the tangent plane at $\bsz^*$.

For a linear $\phi$, we get $\partial\Omega = \{\bsz\mid \theta^\tran\bsz=t\}$,
and we find from Theorem~\ref{thm:steplowerbound}
that $\e(\delta(\bsx))$ is then bounded below by a multiple of $\Vert\theta\Vert_1$
which can be as large as $\sqrt{d}$.
For more general $\phi$, the boundary set $\partial\Omega$ is no longer
an affine flat, the sparsity measure $\Vert\theta^*\Vert_1$ varies spatially over 
$\partial\Omega$, and so does the length $t^*$. 
A large mean dimension, comparable to $\sqrt{d}$, could
arise if $\phi$ has a nonsparse gradient over
an appreciable proportion of $\partial\Omega$.

If the assumption that 
$\lim_{z_j\to\infty}\phi(\bsx_{-j}{:}z_j)>0$ fails,
or if $\lim_{z_j\to-\infty}\phi(\bsx_{-j}{:}z_j)<0$ fails, 
for some value $\bsx_{-j}$,  then
we can no longer find the corresponding point $z_j$.  In that case, the given
value of $j$ and $\bsx_{-j}$ contribute nothing to the numerator
of $\nu(f)$.  The mean dimension can still be large due to
contributions from other values of $\bsx_{-j}$ and from other $j$.
A similar issue came up in \cite{grie:kuo:sloa:2017} where
existence of $z_j$ for every $\bsx_{-j}$ proved not to be
satisfied by an integrand from computational finance, and
also proved not to be necessary for the smoothing effect
of ANOVA to hold.

\subsection{Cusps of general order}

For $d\ge1$ and $\bsx\in[0,1]^d$, 
consider a cusp of order $p>0$ given by
\begin{align}\label{eq:fdp}
f_{d,p}(\bsx) 
= 
\biggl(\,\sum_{j=1}^dx_j-(d-1)\biggr)_+^p 
\end{align}
taking $f_{d,0}(\bsx) =1\{ \sum_{j=1}^dx_j>d-1\}$.
Now $\Vert f_{d,0}\Vert_\hk=\infty$ for $d\ge2$ \cite{variation},
$\Vert f_{d,1}\Vert_\hk=\infty$ for $d\ge3$,
and more generally $\Vert f_{d,p}\Vert_\hk=\infty$
for $d\ge p+2$. The higher the dimension, the greater
smoothness is required to have finite variation.
The boundary $\{\bsx\mid\sum_jx_j=d-1\}$ is not parallel
to any of the coordinate axes, so this integrand is not
QMC-friendly in any way.

These functions are carefully constructed to be among
the simplest with the prescribed level of smoothness.
As a result, we may find their mean dimension analytically.  
%The mean dimension is very close to $d$ for both the kinks ($p=1$),
%and the jumps ($p=0$).
%The reason is that the interesting feature takes place
%within a simplex of volume $1/d!$. Thus more than
%just smoothness affects mean dimension.

\begin{theorem}\label{thm:meandimfdp}
The function $f_{d,p}$ defined above for $\bsx\sim\dustd[0,1]^d$ has mean dimension
\begin{align*}
\nu(f_{d,p}) & = 
d\times \frac{
\frac{\Gamma(2p+1)}{\Gamma(2p+d+1)}
-\bigl(\frac{\Gamma(p+1)}{\Gamma(p+2)}\bigr)^2 \frac{\Gamma(2p+3)}{\Gamma(2p+d+2)}
}
{
\frac{ \Gamma(2p+1)}{\Gamma(2p+d+1)}
-\bigl(\frac{ \Gamma(p+1)}{\Gamma(p+d+1)}\bigr)^2
}. 
\end{align*}
\end{theorem}
\begin{proof}
See Section~\ref{sec:thm:meandimfdp} of the Appendix.
\end{proof}

The functions $f_{d,0}$ have jumps.
Taking $p=0$ in Theorem~\ref{thm:meandimfdp} yields
\begin{align*}
\nu(f_{d,0}) & = 
\frac{d\bigl(\frac{\Gamma(1)}{\Gamma(d+1)}
-\bigl(\frac{\Gamma(1)}{\Gamma(2)}\bigr)^2 \frac{\Gamma(3)}{\Gamma(d+2)}\bigr) 
}{\frac{ \Gamma(1)}{\Gamma(d+1)}
-\bigl(\frac{ \Gamma(1)}{\Gamma(d+1)}\bigr)^2 
}
%\\
%& = 
%\frac{d\bigl(\frac1{d!}-\frac2{(d+1)!}\bigr) 
%}{\frac1{d!}-\frac1{d!^2}
%}\\
%& 
= 
d\times \frac{1-\frac2{d+1}}{1-\frac1{d!}}. 
\end{align*}
Thus $\nu(f_{d,0}) = d-2+o(1)$ as $d\to\infty$.
For kinks, we take $p=1$ in Theorem~\ref{thm:meandimfdp}, getting
\begin{align*}
\nu(f_{d,1}) & = 
\frac{d\bigl(
\frac{\Gamma(3)}{\Gamma(3+d)}
-\bigl(\frac{\Gamma(2)}{\Gamma(3)}\bigr)^2 \frac{\Gamma(5)}{\Gamma(4+d)}\bigr) 
}
{\frac{ \Gamma(3)}{\Gamma(3+d)}
-\bigl(\frac{ \Gamma(2)}{\Gamma(2+d)}\bigr)^2 
}
%\\
%&=\frac{d\Bigl(
%\frac2{\Gamma(3+d)}
%- \frac6{\Gamma(4+d)}\Bigr) 
%}
%{\frac2{\Gamma(3+d)}
%-\bigl(\frac1{\Gamma(2+d)}\bigr)^2 
%}\\
%&=\frac{d\Bigl(
%\frac2{(d+2)!}
%- \frac6{(d+3)!}\Bigr) }
%{\frac2{(d+2)!}
%-\bigl(\frac1{(d+1)!}\bigr)^2 }\\
%&
=d\times \frac{1- \frac3{d+3} }
{1-\frac{d+2}{2(d+1)!}}.
\end{align*}
Therefore $\nu(f_{d,1})=d-3+o(1)$ as $d\to\infty$.
We might reasonbly have guessed that $\nu(f_{d,p})\sim d-p-1$ but we get instead that
$\nu(f_{d,p})\sim d-(4p+2)/(p+1)$ and so even with very large $p$, 
$\lim_{d\to\infty} d-\nu(f_{d,p})$ is not very large.

In this example we see that even when the cusp is very smooth, the integrand
does not end up dominated by its low dimensional ANOVA components.
A key difference between this example and the ridge functions defined over
Gaussian random vectors is that these cusp functions are zero apart from
a set of volume $1/d!$.  As $d$ increases the integrands become ever
more dominated by a rare event. The Gaussian integrands by contrast attained
somewhat higher mean dimension for large $t$ but $\Pr( \theta^\tran\bsx>t)$
remained constant as $d$ increased. 

\section{Preintegration}\label{sec:preintegration}

In preintegration we integrate over one component $x_\ell$ either
in closed form or by a univariate quadrature rule that has negligible
error.  For $\bsx\sim\dnorm(0,I)$, the preintegrated function is
$$
\bar f_\ell(\bsx) = \int_{-\infty}^\infty \varphi(x_\ell)f(\bsx)\rd x_\ell.
$$
Preintegrating over multiple components yields
$\bar f_u =\int_{\real^{|u|}}f(\bsx)\prod_{j\in u}\varphi(x_j)\prod_{j\in u}\rd x_j$,
for $u\subset1{:}d$.
Preintegration for $\bsx\sim\dustd[0,1]^d$ is similar.

The function $\bar f_\ell$ is intrinsically $d-1$ dimensional but for notational
convenience we leave it as a function of $d$ arguments that is constant with respect to $x_\ell$.
Preintegration can increase the smoothness of the integrand \cite{grie:kuo:leov:sloa:2018}
making it conform to the sufficient conditions used in (R)QMC and also those
used for sparse grid methods \cite{bung:grie:2004}.

Here we show some elementary properties about preintegration
including its effect on the ANOVA decomposition and mean dimension.
We also show that preintegration preserves the ridge function property
and any H\"older conditions.

\begin{proposition}\label{prop:easyone}
Let $\bsx\in\real^d$ have the $\dnorm(0,I)$ distribution.
If $f(\bsx)=g(\bsx^\tran\theta)$ for a unit vector $\theta$
then $\bar f_\ell(\bsx)$, for $\ell\in1{:}d$ is also a ridge function.  
If $g$ satisfies a H\"older condition with constant $C$ and
exponent $\alpha\in(0,1]$, then so does $\bar f_\ell$, with the same $\alpha$ and
$C_\ell = (1-\theta_\ell^2)^{1/2}C$.
\end{proposition}
\begin{proof}
If $|\theta_\ell|=1$ then $\bar f_\ell$ is constant and hence trivially
a ridge function and also H\"older continuous.
For $|\theta_\ell|<1$, define
$\theta^*_{\ell} = \theta_{-\ell}{:}0_\ell/({1-\theta_\ell^2})^{1/2}$. Then
\begin{align*}
\bar f_{\ell}(\bsx) 
&=\int_{-\infty}^\infty \varphi(x_\ell)g\bigl(\theta_\ell x_\ell + 
({1-\theta_\ell^2})^{1/2}\theta_{\ell}^{*\tran}\bsx\bigr)\rd\bsx
\equiv \bar g_\ell( \theta_\ell^{*\tran}\bsx),\quad\text{where}\\
\bar g_\ell(y) &= \int_{-\infty}^\infty \varphi(x)g(\theta_\ell x+(1-\theta_\ell^2)^{1/2}y)\rd x.
\end{align*}
This establishes that $\bar f_\ell$ is a ridge function.
Next for $y,y'\in\real$,
$|\bar g_\ell(y')-\bar g_\ell(y)|
\le (1-\theta_\ell^2)^{1/2}|g(y')-g(y)|$.
%\begin{equation*}
%|\bar g_\ell(z')-\bar g_\ell(z)|
%\le (1-\theta_\ell^2)^{1/2}|\bar g_\ell(z')-\bar g_\ell(z)|.
%\end{equation*}
\end{proof}

The mean dimensions  before and after preintegration are
$$
\nu(f) = 
\frac{\sum_{u\subseteq1{:}d}|u|\sigma^2_u}
{\sum_{u\subseteq1{:}d}\sigma^2_u}\quad\text{and}\quad
\nu(\bar f_\ell) = 
%\frac{\sum_{u\subseteq1{:}d,\ell\not\in u}|u|\sigma^2_u}
%{\sum_{u\subseteq1{:}d,\ell\not\in u}\sigma^2_u}
\frac{
\sum_{u}|u|\sigma^2_u-\sum_{u:\ell\in u}|u|\sigma^2_u
}
{
\sum_{u}\sigma^2_u
-\sum_{u:\ell\in u}\sigma^2_u
}.
$$
Preintegration over $x_\ell$ removes $|u|\sigma^2_u$ from the numerator and $\sigma^2_u$
from the denominator, for each $u$ with $\ell\in u$.
The greatest mean dimension reductions come from preintegrating variables that
contribute to large high order variance components. Preintegrating
a variable that only contributes to $f$ additively will increase mean
dimension (unless $f$ is entirely additive), although such preintegration
may well produce a useful variance reduction.

After some algebra, preintegration over $\bsx_u$ reduces mean dimension if
\begin{align}\label{eq:reduces}
\frac{\sum_{v:v\cap u = \emptyset}|v|\sigma^2_v}{\sigma^2-\olt^2_u}
<
\frac{\sum_{v:v\cap u\ne \emptyset}|v|\sigma^2_v}{\olt^2_u}.
\end{align}
The left hand side of~\eqref{eq:reduces} is $\nu(\bar f_u)$
and the right hand side is $\nu(f-\bar f_u)$.
To take an extreme example, if $f-\bar f_u$ is additive then
preintegration cannot reduce mean dimension. 
%Similarly, if $x_\ell$ only enters $f$ additively, meaning that
%$\sigma^2_v=0$ if both  $\ell\in v$ and $|v|\ge2$ hold, then
%preintegrating $x_\ell$ cannot reduce mean dimension.
Conversely, if $\bar f_u$ is additive, then preintegration over $\bsx_u$
reduces mean dimension to one.

\subsection{Preintegrated step function}
As a worked example we consider
preintegration of a ridge step function $f(\bsx)=g(\theta^\tran\bsx)$
for $g(y) = 1\{y>t\}$ for some threshold $t$ and $\bsx\sim\dnorm(0,I)$.
For special cases, such as $t=0$ and $\theta = \one_d/\sqrt{d}$
we can get more precise results.

The preintegrated function $\bar f_\ell$ is a ridge function with
\begin{align*}
\bar g_\ell(y) &=
\int_{-\infty}^\infty \varphi(x)
1\{\theta_\ell x+(1-\theta_\ell^2)^{1/2} y>t\}\rd x
 = \Phi\Bigl( 
\frac{(1-\theta_\ell^2)^{1/2}y-t}{\theta_\ell}
\Bigr).
%\bigl(\bigr)/\sqrt{\theta_\ell}\bigr)
\end{align*}
Differentiating
$$
\bar g_\ell'(y) = \varphi\Bigl(
\frac{(1-\theta_\ell^2)^{1/2}y-t}{\theta_\ell}\Bigr)
\frac{(1-\theta_\ell^2)^{1/2}}{\theta_\ell}
$$
and so this ridge function is Lipschitz with 
$C_\ell = \varphi(0){(1-\theta_\ell^2)^{1/2}}/{|\theta_\ell|}$
%$$C_\ell = \varphi(0)\frac{(1-\theta_\ell^2)^{1/2}}{|\theta_\ell|}$$
leading to a mean dimension for $f$ of no more than
\begin{align*}
\Bigl(\frac{C_\ell}{\sigma}\Bigr)^2
=
\frac{\varphi(0)^2}{\Phi(t)\Phi(-t)}\frac{1-\theta_\ell^2}{\theta_\ell^2}.
\end{align*}
This bound is minimized by taking $\ell = \arg\max_j|\theta_j|$.
While the ridge function formed by preintegrating the step function
is infinitely differentiable and hence
much smoother than the kink $(\theta^\tran\bsx-t)_+$, it 
could have a very large Lipschitz constant due to the
presence of $|\theta_\ell|$ in the denominator.
While the preintegrated function has a large Lipschitz constant,
the step function without preintegration was not Lipschitz at all.

For the case with $\theta_\ell=\pm1/\sqrt{d}$ the bound becomes
$$\frac{\varphi(0)^2}{\Phi(t)\Phi(-t)}\frac{1-\theta_\ell^2}{\theta_\ell^2}
=\frac1{2\pi}\frac{d-1}{\Phi(t)\Phi(-t)}.
$$
This bound is only below $d-1$  for $t$ near zero.
For $t=0$ we get a bound of about $0.64(d-1)$.

As remarked above these bounds can be conservative.
The step function has a simple enough discontinuity
that we can explore the mean dimension of it
under preintegration.  
%The problem here is that the Lipschitz bound is very conservative.
%The derivative of $\bar g_\ell$ is negligible apart from a set
%of probability $O(1/\sqrt{d})$ when $z\sim\dnorm(0,1)$
%and $\theta_\ell=1/\sqrt{d}$.

\begin{theorem}\label{thm:preintegrated}
For $\bsx\sim\dnorm(0,I_d)$, 
let $f(\bsx) = 1\{\theta^\tran\bsx>t\}$ where $\Vert\theta\Vert=1$. 
Choose $\ell$ with $\theta_\ell\ne 0$ and let $\bar f_\ell$ be $f$
preintegrated  over $x_\ell$. 
Then 
\begin{align}\label{eq:preintfirst}
\nu(\bar f_\ell) = 
\frac{2\varphi(t) 
\sum_{j\ne\ell}\int_{a_1}^{a_2(j)}\frac{\varphi(tx)}{1+x^2}\rd x 
}
{
\Phi(t)\Phi(-t) -2\varphi(t)\int_0^{a_1}\frac{\varphi(tx)}{1+x^2}\rd x 
}. 
\end{align}
If $t=0$, then 
\begin{align}\label{eq:preintsecond}
\nu(\bar f_\ell) =  \frac{ \sum_{j\ne\ell}
\bigl(\tan^{-1}(a_2(j))-\tan^{-1}(a_1)\bigr)}
{\pi/4 -\tan^{-1}(a_1)}, 
\end{align}
where  $a_1=\theta_\ell/(2-\theta_\ell^2)^{1/2}$
and $a_2(j) = (\theta_j^2+\theta_\ell^2)^{1/2}/(2-\theta_j^2-\theta_\ell^2)^{1/2}$. 
If also $\theta_j=\theta_\ell=1/\sqrt{d}$, then 
\begin{align}\label{eq:preintthird}
\nu(\bar f_\ell) &= 
\frac{(d-1)[ \tan^{-1}( (d-1)^{-1/2}) 
-\tan^{-1}( (2d-1)^{-1/2})]}
{\pi/4 -\tan^{-1}( (d-1)^{-1/2})}
%=\frac\pi4\sqrt{d} + O(d^{-1/2}) 
\end{align}
so $\nu(\bar f_\ell) =(\pi/4)\sqrt{d} + O(d^{-1/2})$
as $d\to\infty$. 
\end{theorem}

If we had not preintegrated $1\{\sum_jx_j/\sqrt{d}>0\}$
the mean dimension  would have been asymptotic to $(2/\pi)\sqrt{d}$
from~\eqref{eq:meandimobliquejump}.
For the step function on a least sparse $\theta$,
preintegration brings a small reduction
in variance, an enormous improvement in smoothness,
but a small increase in the mean dimension.  That increase
is unimportant because neither $f$ nor $\bar f_\ell$
has a small mean dimension when $d$ is large.
It is more important that the $\sqrt{d}$ rate has not changed.

Things are very different if one of the $|\theta_\ell|$ is large
and $\Vert\theta\Vert_\infty$ is bounded away from zero as $d\to\infty$.
Preintegrating that variable leads to a Lipschitz constant of 
$C_\ell =\varphi(0)(1-\Vert\theta\Vert_\infty^2)^{1/2}/\Vert\theta\Vert_\infty$
and a mean dimension  of
$$
\nu(f) \le 
\frac{\varphi(0)^2}{\Phi(t)\Phi(-t)}\frac{1-\Vert\theta\Vert_\infty^2}{\Vert\theta\Vert_\infty^2}
\le\frac{\varphi(0)^2}{\Phi(t)\Phi(-t)}\Vert\theta\Vert_\infty^{-2}.
$$
In this case the mean dimension remains bounded as $d\to\infty$.
Had we not preintegrated, the mean dimension would have been 
bounded below by a multiple of $\Vert\theta\Vert_1$ which could
diverge. For instance with $\theta_1=1/2$ and $\theta_j=(2(d-1))^{-1/2}$
we get $\nu(f)$ bounded below by a multiple of $\sqrt{d}$ while
$\nu(\bar f_1)$ is bounded above by a constant as $d\to \infty$.
The finance example in~\cite{grie:kuo:leov:sloa:2018} involves preintegration
of an extremely important variable and it lead to a great improvement
in QMC integration.

\subsection{Smoothing by dimension increase}
An earlier smoothing method \cite{mosk:cafl:1996} replaces step discontinuities 
by `beveled edges' of some half-width $\delta>0$.  
For a set $\Omega\subset\real^d$ with a well-behaved boundary,
they replace the integral of the indicator function $1\{\bsx\in \Omega\}$ 
by that of a function which is $0$ if $\bsx$ is farther than $\delta$
from $\Omega$, is $1$ if $\bsx$ is farther than $\delta$ from $\Omega^c$ 
and is a linear function of the signed distance from $\bsx$ to $\partial\Omega$
in between.  They have a similar smoothed rejection technique 
that involves replacing the discontinuous function over $[0,1]^d$
by a smooth one over $[0,1]^{d+1}$.  See also~\cite{wang:2000}. 
We won't compare these to preintegration beyond noting 
how interesting it is that dimension increase and dimension 
reduction have both been proposed as methods to handle discontinuous integrands. 

\section{Numerical examples}\label{sec:numerical}

We can estimate $\nu(f)$ for $\theta = \one_d/\sqrt{d}$
via the three dimensional integral in equation~\eqref{eq:nufleastsparse}.
To estimate that integral
we used Sobol' sequences in $[0,1]^d$ \cite{sobo:1967:tran} with 
direction numbers from \cite{joe:kuo:2008}
with data from Nuyens' magic point shop
described in \cite{kuo:nuye:2016}.
The points were given a nested uniform scramble as described in \cite{rtms}
and then transformed via $\Phi^{-1}(\cdot)$ into Gaussian random vectors.
For each dimension we considered, we did five independent replicates.
%> source("figs.R"); system.time( figgausjumps( dlist=2^c(0:27), m1=12,reps=5 ) )
%   user  system elapsed 
%  3.950   0.313   4.332 
%> source("figs.R"); system.time( figgausjumps( dlist=2^c(0:27), m1=18,reps=5 ) )
%   user  system elapsed 
%274.126  34.725 311.673 
%> source("figs.R"); system.time( figgausjumps( dlist=2^c(0:27), m1=20,reps=5 ) )
%    user   system  elapsed 
%1116.791  157.711 3868.728 

%> source("figs.R"); system.time( figgauskinks( dlist=2^c(0:27), m1=15,reps=5 ) )
%   user  system elapsed 
% 49.935   5.625  55.643 

\begin{figure}
\centering 
\includegraphics[width=\hsize]{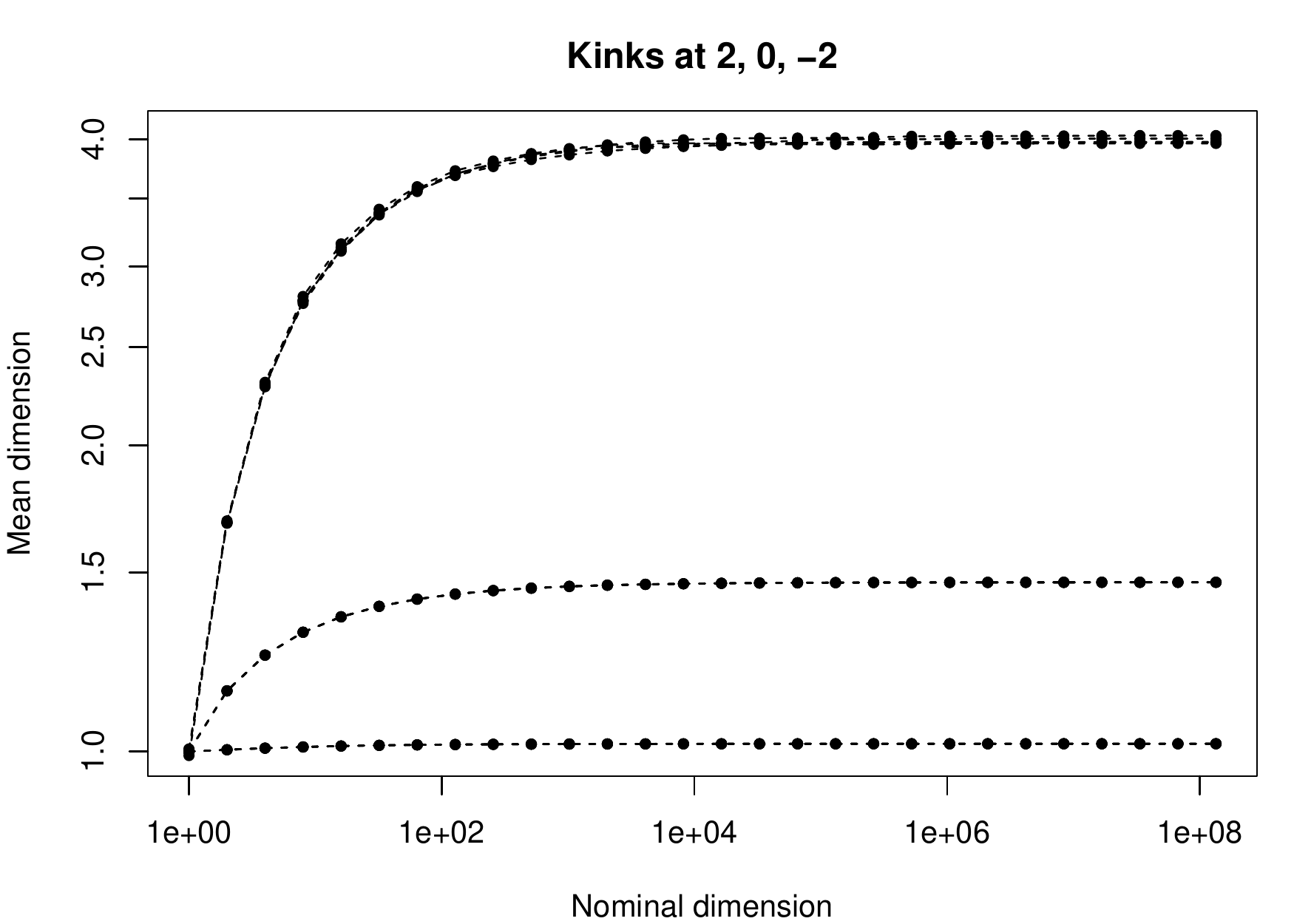}
\caption{\label{fig:gauskinks}
Computed mean dimension for $f(\bsx) = \max( \theta^\tran\bsx-t,0)$,
with $\theta_j=1/\sqrt{d}$ versus nominal dimension $d$. 
From top to bottom the thresholds are $t=2,0,-2$. 
There were 5 independent computations with using $2^{15}$ scrambled 
Sobol' points each. 
}
\end{figure}

Figure~\ref{fig:gauskinks} shows mean dimensions computed for 
$f(\bsx) = \max( \sum_{j=1}^dx_j/\sqrt{d}-t,0)$, a kink function,
for $t\in\{2,0,-2\}$.    All five replicates are plotted for each 
threshold; they overlap considerably.  For $t=0$ we established 
that $\nu(f)\le2.933$ in Section~\ref{sec:kinkfunctionbound}. 
The mean of five replicated $\nu(f)$ values for $d=2^{27}$ was 
$1.47$ almost exactly half of the bound with a standard 
error of $0.00014$.  
The bound in Section~\ref{sec:kinkfunctionbound}
gives about $175.5$ for $t=2$ which is much larger than
the computed values. It also gives just over $1.041$
for $t=-2$.

\begin{figure}
\centering 
\includegraphics[width=\hsize]{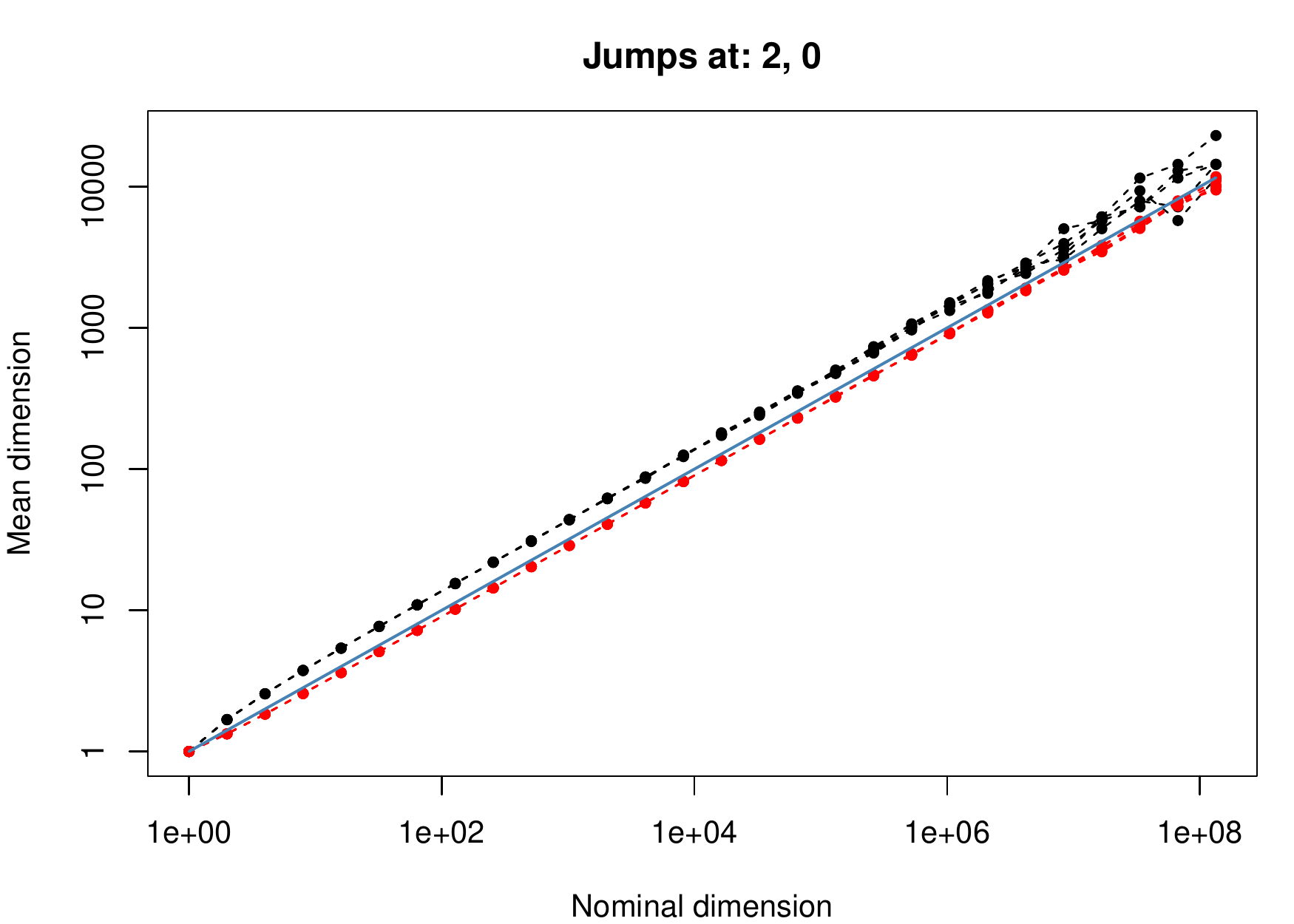}
\caption{\label{fig:gausjumps}
Computed mean dimension for $f(\bsx) = 1\{\theta^\tran\bsx>t\}$,
with $\theta_j=1/\sqrt{d}$ versus nominal dimension $d$. 
From top to bottom the thresholds are $t=2,0$.
There were 5 independent computations with using $2^{20}$ scrambled 
Sobol' points each. There is a reference line $y=\sqrt{d}$ in between
the two sets of curves.
}
\end{figure}

Figure~\ref{fig:gausjumps} shows mean dimensions computed for 
$f(\bsx) = 1\{ \sum_{j=1}^dx_j/\sqrt{d}>t\}$, a jump function,
for $t\in\{2,0\}$. The mean dimension is the same for $t$
as for $-t$, so we do not include $t=-2$.   All five replicates are plotted for each 
threshold; they overlap considerably for $d\le 10^6$.  
For larger $d$,  fluctuations are visible especially for $t=2$.
The estimated mean dimensions are very nearly parallel
to $\sqrt{d}$ over this range.

\section{Conclusions}\label{sec:conclusions}

Integrands formed as ridge functions over
Gaussian random variables $\bsx\sim\dnorm(0,I)$
can have bounded mean dimension as the nominal dimension
increases.  It suffices for them to be Lipschitz functions
of $\theta^\tran\bsx$ for a unit vector~$\theta$.

Ridge functions are simple enough that they can be integrated
directly via one dimensional quadrature, and in some cases,
by closed form expressions, yielding good test functions. 
In applications, an integrand may
be close to a ridge function without the user being aware of
it. Constantine \cite{cons:2015} finds that many functions
in engineering applications are well approximated by ridge functions.
Some of our findings are for specific functions such 
as $(\theta^\tran\bsx-t)_+$ or $1\{\theta^\tran\bsx>t\}$ 
and it remains to see how generally they apply
to other kinks and jumps.

Suppose that $f$ is approximately a ridge function of low mean dimension.
We write $f(\bsx) = g(\theta^\tran\bsx) + \err(\bsx)$.
Then under scrambled net sampling, the MSE is
\begin{align*}
\var\Biggl( \frac1n\sum_{i=1}^n \bigl( g(\theta^\tran\bsx_i)+\err(\bsx_i)\bigr)\Biggr)
&\le2\var\Biggl( \frac1n\sum_{i=1}^ng(\theta^\tran\bsx_i) \Biggr)
+2\var\Biggl( \frac1n\sum_{i=1}^n \err(\bsx_i)\Biggr)\\
&
\le2\var\Biggl( \frac1n\sum_{i=1}^ng(\theta^\tran\bsx_i) \Biggr)
+2\Gamma \frac{\var(\err(\bsx))}n,
\end{align*}
where $\Gamma$ is the largest gain coefficient \cite{snxs}.
The factor of $2$ is a conservative upper bound.
The first term benefits from low mean dimension of ridge
functions and the smoothing effect of the ANOVA.
%The factor of $2$ conservatively describes a setting where $\err(\bsx)$ is
%perfectly correlated with $g(\theta^\tran\bsx)$ and in that case we could 
%redefine $g$ and have $\err=0$.

In projection pursuit regression \cite{frie:steu:1981}, a 
high dimensional function is approximated by a sum of a 
small number of ridge functions.  
Single layer (not deep) neural networks approximate
a function by a linear combination of smooth ridge functions
\cite{cybe:1989}.
Historically those ridge functions were smooth CDFs
like $g(y) = (1+\exp(-y))^{-1}$ and more recently
the positive part function $g(y) = \max(y,0)$
also called a rectified linear unit (relu) has been prominent.
Both of these $g(\cdot)$ are Lipschitz.  
Those models are often good 
approximations to real world phenomena. 
They are usually fit to noisy data but noise
is not a critical part of them being a good fit.

Suppose now that $f(\bsx) = \sum_{j=1}^Jf_j(\bsx)$,
where $f_1,\dots,f_{J-1}$ are ridge functions and
$f_J$ is a residual function with a small mean square.
Then under RQMC sampling
$\var(\hat\mu) \le J\sum_{j=1}^J\var(\hat\mu_j)$
where $\hat\mu_j$ is the average of $f_j(\bsx_i)$
over an RQMC sample $\bsx_i$.
The factor $J$ is extremely
conservative as it allows for perfect correlations among all 
$J$ integration errors.

We have not addressed whether it is realistic to expect $g$
to remain constant as $d\to\infty$.  A full discussion of 
that point is beyond the scope of this article.  Instead 
we make a few remarks. 

If we think of Brownian motion with $d$ time steps to time $T=1$
then under the standard construction, the end point 
is $B(T)=(1/\sqrt{d})\sum_{j=1}^dx_j$. 
In this instance making $\theta$ a unit vector is a good generalization 
of infill asymptotics and a function of $B(T)$ or $B(\lambda T)$
for $0<\lambda<1$ takes 
on the form $g(\theta^\tran\bsx)$ for $\Vert\theta\Vert=1$. 
If instead, we consider Brownian motion with $d$ time steps to time $T=d$
then under the standard construction, the endpoint is $B(T)=\sum_{j=1}^dx_j$. 
We might model that via $f(\bsx) = g( \sqrt{d}\theta^\tran\bsx)$. 
Introducing $\sqrt{d}$ within $g(\cdot)$ multiplies any Lipschitz bound for $g$
by $\sqrt{d}$ and then raises the  upper bound on $\sum_j\olt_j^2$ by a 
factor of $d$.  Whatever effect this has on $\nu(f)$ depends on how 
introducing $\sqrt{d}$ within $g(\cdot)$ affects $\sigma^2$, the variance of $f$. 
The variance might also increase by a factor of $d$, leaving the mean 
dimension invariant to $d$. For instance, that would happen 
for $f(\bsx) = (\sum_jx_j-t)_+$. 
If instead the variance remains remains nearly constant, then 
the mean dimension could grow with $d$. 
For instance if $f(\bsx) = \Phi( \sum_jx_j)$, then for large $d$
it is like a Heaviside function applied to $(1/\sqrt{d})\sum_jx_j$
and the mean dimension will grow like $\sqrt{d}$.

\section*{Acknowledgments}
This work was supported by the U.S.\ 
National Science Foundation under grants IIS-1837931 and DMS-1521145.

\goodbreak
\bibliographystyle{siamplain}
\bibliography{qmc}

\vfill\eject
\section{Appendix}\label{sec:appendix}

\subsection{Upper bound for jumps}
\label{sec:thm:stepupperbound}
\begin{proof}
Here we prove Theorem~\ref{thm:stepupperbound}.
If $\theta_k=0$ then $\olt^2_k=0$ too.  
We may suppose that any such $x_k$ have been removed from the model. 
Then 
\begin{align*}
\olt^2_k 
& = \frac12\e\Bigl( \bigl(1\{y+x>t\} - 1\{y+z>t\}\bigr)^2\Bigr)\\
& = \frac12\e\Bigl( |1\{y+x>t\} - 1\{y+z>t\}|\Bigr) 
\end{align*}
where $y\sim\dnorm(0,1-\theta_k^2)$ and $x, z\sim\dnorm(0,\theta_k^2)$
are all independent. Next, for any $\epsilon>0$
\begin{align*}
2\olt^2_k 
&\le \Pr( |y+x-t| <\epsilon )  + \Pr( |z-x|>\epsilon)\\
& = \Phi( -t+\epsilon ) - \Phi( -t-\epsilon ) + 2\Phi\Bigl(\frac{-\epsilon}{\sqrt{2}|\theta_k|}\Bigr). 
\end{align*}
As a result 
\begin{align*}
\nu(f) &\le \frac1{2\Phi(t)\Phi(-t)}\sum_{j=1}^d 
\Phi( -t+\epsilon_k ) - \Phi( -t-\epsilon_k ) + 2\Phi\Bigl(\frac{-\epsilon_k}{\sqrt{2}|\theta_k|}\Bigr). 
\end{align*}
Taking $\epsilon_k=\eta|\theta_k|$, 
\begin{align*}
\nu(f) &\le 
\frac1{2\Phi(t)\Phi(-t)}
\biggl( 2d\Phi\Bigl( -\frac{\eta}{\sqrt{2}}\Bigr)+\sum_{j=1}^d 
\Phi( -t+\eta|\theta_k| ) - \Phi( -t-\eta|\theta_k| ) \biggr)\\
&\le \frac1{\Phi(t)\Phi(-t)}
\biggl( \frac{d}{\eta/\sqrt{2}}\varphi\Bigl( -\frac{\eta}{\sqrt{2}}\Bigr)+
\eta\varphi(0)\Vert \theta\Vert_1\biggr)\\
&\le \frac1{\Phi(t)\Phi(-t)\sqrt{2\pi}}
\biggl( \frac{\sqrt{2}d}{\eta}
\exp\Bigl(-\frac{\eta^2}4\Bigr) 
+\eta\Vert \theta\Vert_1\biggr). 
\end{align*}
Choosing $\eta=2\sqrt{\log(d/\Vert\theta\Vert_1)}$, 
\begin{align*}
\nu(f) &\le 
 \frac{\Vert\theta\Vert_1}{\Phi(t)\Phi(-t)\sqrt{2\pi}}
\Bigl( \frac{\sqrt{2}}{\eta}+\eta\Bigr) 
\end{align*}
To conclude, $1\le\Vert\theta\Vert_1\le\sqrt{d}$, so for $d\ge2$, 
$\eta\ge 2\sqrt{ \log(2/\sqrt{2})}>1$. 
\end{proof}

\subsection{A bivariate Gaussian probability lower bound}
\label{sec:lem:bivgauslowerbound}
\begin{proof}
Here we prove Lemma~\ref{lem:bivgauslowerbound}.
For $\eta>0$,
\begin{align*}
&\phantom{\ge}\ \Pr( x>t,y<t) \\
&\ge \Pr( x > t+\eta, y< t-\eta)\\
&=\frac1{2\pi(1-\rho^2)^{1/2}}\int_t^{t+\eta}\int_{t-\eta}^t\exp 
\Bigl(-\frac12[y_1^2-2\rho y_1y_2+y_2^2]/(1-\rho^2) 
\Bigr)\rd y_2\rd y_1\\
&\ge\frac{\eta^2}{2\pi(1-\rho^2)^{1/2}}
\exp\Bigl(-\frac12[(t_1+\eta)^2-2\rho (t_1+\eta)(t_2-\eta)+(t_2-\eta)^2]/(1-\rho^2) 
\Bigr),
\end{align*}
because with $\rho \ge0$ and $t\ge0$, the bivariate normal probability density function 
is minimized over $[t,t+\eta]\times[t-\eta,t]$ at $(t+\eta,t-\eta)$. 
Simplifying this expression and then choosing $\eta =\sqrt{1-\rho}$,
\begin{align*}
\Pr( x>t,y<t) &\ge 
\frac{\eta^2}{2\pi(1-\rho^2) ^{1/2}}
\exp\Bigl(-\frac{t^2}{1+\rho}
-\frac{\eta^2}{1-\rho}\Bigr)\\
&=\frac1{2\pi}\Bigl(\frac{1-\rho}{1+\rho}\Bigr)^{1/2}
\exp\Bigl(-\frac{t^2}{1+\rho}-1\Bigr). 
\end{align*}
\end{proof}

\subsection{Lower bound for jumps}
\label{sec:thm:steplowerbound}
\begin{proof}
Here we prove Theorem~\ref{thm:steplowerbound}.
Letting $y_1=\theta^\tran\bsx$ and 
$y_2 = x+\theta_j(z_j-x_j)$ we get 
$$
\begin{pmatrix}
y_1\\
y_2 
\end{pmatrix}
\sim 
\dnorm\biggl(
\begin{pmatrix}
0\\
0 
\end{pmatrix}
, \begin{pmatrix}
1 & \rho_j\\
\rho_j& 1 
\end{pmatrix}
\biggr),\quad \text{for $\rho_j =1-\theta_j^2.$}
$$
Now $\olt^2_j = \Pr( y_1>t,y_2<t)$. 
From Lemma~\ref{lem:bivgauslowerbound}.
\begin{align*}
\olt^2_j 
& \ge
\frac1{2\pi}\Bigl(\frac{1-\rho_j}{1+\rho_j}\Bigr)^{1/2}
\exp\Bigl(-\frac{t^2}{1+\rho_j}-1\Bigr)\\
&= \frac1{2\pi}
\frac{|\theta_j|}{(2-\theta_j^2)^{1/2}}
\exp\Bigl(-\frac{t^2}{2-\theta_j^2}-1\Bigr)\\
&\ge \frac{|\theta_j|}{2^{3/2}\pi}\exp(-t^2-1). 
\end{align*}
Summing over $j\in1{:}d$ and dividing
by $\sigma^2 = \Phi(t)\Phi(-t)$ completes the proof.
\end{proof}

\subsection{Upper bound for kinks}
\label{sec:thm:meandimfdp}
\begin{proof}
Here we prove Theorem~\ref{thm:meandimfdp}. 
First for $j\in 1{:}d$, 
\begin{align*}
\int_0^1 f_{d,p}(\bsx)\rd x_j 
%&= \frac1{p+1}\biggl(x_d+\sum_{j=1}^{d-1}x_j-(d-1)\biggr)_+^{p+1} \Biggm|_{x_d=0}^{x_d=1}\\
%&= \frac1{p+1}\biggl(1+\sum_{j=1}^{d-1}x_j-(d-1)\biggr)_+^{p+1} \\
&=\frac1{p+1}f_{d-1,p+1}(\bsx_{-j}). 
\end{align*}
Applying this result $|u|$ times,  for $u\subsetneq1{:}d$, yields 
\begin{align}\label{eq:ifdp}
\int_{[0,1]^{|u|}}f_{d,p}(\bsx)\rd x_u 
&
%=\frac{ f_{d-|u|,p+|u|}(\bsx_{-u})}{(p+1)(p+2)\cdots(p+|u|)}
=
\frac{ \Gamma(p+1)}{\Gamma(p+|u|+1)} f_{d-|u|,p+|u|}(\bsx_{-u}).  
\end{align} 
%\begin{align}\label{eq:ifdpc}
%g_{d,p,u}(\bsx)\equiv 
%\int_{[0,1]^{d-|u|}}f_{d,p}(\bsx)\rd x_{-u }
%=
%\frac{ \Gamma(p+1)}{\Gamma(p+d-|u|+1)}
%f_{|u|,p+d-|u|}(\bsx_{u}). 
%\end{align}
%The function $g_{d,p,u}$ is the best $L^2$ approximation to $f_{d,p}$
%using only $\bsx_u$. 
Applying~\eqref{eq:ifdp} formally for $u=1{:}d$ gives 
$$
\mu_{d,p}\equiv \int_{[0,1]^d}f_{d,p}(\bsx)\rd\bsx 
=\frac{ \Gamma(p+1)}{\Gamma(p+d+1)} f_{0,p+d}(\bsx_{\emptyset}). 
$$
We can find more rigorously that 
\begin{align*}
\mu_{d,p} 
&= \int_0^1 
\frac{\Gamma(p+1)}{\Gamma(p+d)}  f_{1,p+d-1}(x_d)\rd x_d=
\frac{\Gamma(p+1)}{\Gamma(p+d)} \int_0^1x^{p+d-1}\rd x 
=\frac{ \Gamma(p+1)}{\Gamma(p+d+1)}, 
\end{align*}
and so we get the correct answer from a 
convention that $f_{0,p+d}(\bsx_\emptyset)=1$. 
The variance of $f_{d,p}$ is 
$$
\sigma^2_{d,p}=
\mu_{d,2p}-\mu_{d,p}^2=
\frac{ \Gamma(2p+1)}{\Gamma(2p+d+1)}
-\Bigl(\frac{ \Gamma(p+1)}{\Gamma(p+d+1)}\Bigr)^2. 
$$

For $f_{d,p}$, we get a Sobol' index of 
\begin{align*}
\olt^2_d & = \frac12 \int 
\biggl( f_{d,p}(\bsx) -\Bigl( z_d + \sum_{j=1}^{d-1} x_j -(d-1)\Bigr)_+^p\rd\bsx\rd z_d\biggr)^2\\
&=\mu_{d,2p}
 -\int f_{d,p}(\bsx) \Bigl( z_d + \sum_{j=1}^{d-1} x_j -(d-1)\Bigr)_+^p\rd\bsx\rd z_d\\
&=\mu_{d,2p} -\Bigl(\frac{\Gamma(p+1)}{\Gamma(p+2)}\Bigr)^2 
\int f_{d-1,p+1}(\bsx_{-d}) ^2\rd\bsx_{-d}\\
&=\mu_{d,2p}-\Bigl(\frac{\Gamma(p+1)}{\Gamma(p+2)}\Bigr)^2\mu_{d-1,2p+2}\\
&= \frac{\Gamma(2p+1)}{\Gamma(2p+d+1)}
-\Bigl(\frac{\Gamma(p+1)}{\Gamma(p+2)}\Bigr)^2 \frac{\Gamma(2p+3)}{\Gamma(2p+d+2)}. 
\end{align*}
Now because $\olt^2_j=\olt^2_d$ by symmetry  for all $j\in1{:}d$, we get 
\begin{align*}
\nu(f_{d,p}) & = 
 \frac{d\Bigl(
 \frac{\Gamma(2p+1)}{\Gamma(2p+d+1)}
 -\bigl(\frac{\Gamma(p+1)}{\Gamma(p+2)}\bigr)^2 \frac{\Gamma(2p+3)}
 {\Gamma(2p+d+2)}\Bigr) }
 {
 \frac{ \Gamma(2p+1)}{\Gamma(2p+d+1)}
 -\bigl(\frac{ \Gamma(p+1)}{\Gamma(p+d+1)}\bigr)^2 
 }. 
\end{align*}
\end{proof}

\subsection{Mean dimension of preintegrated step functions}
\label{sec: thm:preintegrated}
\begin{proof}
Here we prove Theorem~\ref{thm:preintegrated}.
We will use 
\begin{align}
\int_{-\infty}^\infty \Phi( a+bx)\varphi(x)\rd x & = \Phi\biggl(\frac{a}{\sqrt{1+b^2}}\biggr), \quad\text{and}
\label{eq:gauscdfpdf}
\\
\int_{-\infty}^\infty \Phi( a+bx)^2\varphi(x)\rd x & = \Phi\biggl(\frac{a}{\sqrt{1+b^2}}\biggr) 
-2T\biggl(
\frac{a}{\sqrt{1+b^2}}, \frac{1}{\sqrt{1+2b^2}}\biggr),\quad\text{where}
\label{eq:gauscdf2pdf}\\
T(h,a) & = \varphi(h)\int_0^a\frac{\varphi(hx)}{1+x^2}\rd x.\notag
\end{align}
These are formulas 10,010.8 and 20,010.4, respectively,
from \cite{owen:1980}.

Recalling that $\theta_\ell\ne0$,
$$
\bar f_\ell(\bsx) 
= \Phi\biggl(
\frac{(1-\theta_\ell^2)^{1/2}\theta^{*\tran}_\ell\bsx-t}{\theta_\ell}
\biggr)
= 
\Phi\biggl(\frac{\sum_{k\ne\ell}\theta_kx_k-t}{\theta_\ell}\biggr)
$$
So for $j\ne \ell$, letting $\gamma_j=\sqrt{1-\theta_j^2-\theta_\ell^2}$,
\begin{align*}
\olt^2_j & = \frac12
\e\Biggl(\biggl[
\Phi\biggl(\frac{\gamma_jy_j + \theta_jx_j-t}{\theta_\ell}\biggr)
-\Phi\biggl(\frac{\gamma_jy_j+\theta_jz_j-t}{\theta_\ell}\biggr)
\biggr]^2
\Biggr)\\
& =
\e\bigl( \bar f_\ell(\bsx)^2\bigr) -
\e\biggl(
\Phi\biggl(\frac{\gamma_jy_j + \theta_jx_j-t}{\theta_\ell}\biggr) 
\Phi\biggl(\frac{\gamma_jy_j + \theta_jz_j-t}{\theta_\ell}\biggr) 
\biggr),
\end{align*}
where $x_j,y_j,z_j$ are independent $\dnorm(0,1)$ random variables.

First, from~\eqref{eq:gauscdf2pdf}
\begin{align*}
\e\bigl( \bar f_\ell(\bsx)^2\bigr) 
& = \int_{-\infty}^\infty 
\Phi\biggl(
\frac{(1-\theta_\ell^2)^{1/2} z-t}{\theta_\ell}
\biggr)^2\rd z
%&=\Phi\biggl(
%\frac{-t/\theta_\ell}{1/\theta_\ell}\biggr) 
%-2T\biggl(\frac{-t/\theta_\ell}{1/\theta_\ell},
%\frac{\theta_\ell}{(2-\theta_\ell^2)^{1/2}}\biggr)\\
=
\Phi\bigl(-t\bigr)-2T\biggl(-t,\frac{\theta_\ell}{(2-\theta_\ell^2)^{1/2}}\biggr).
\end{align*}
Next, applying~\eqref{eq:gauscdfpdf} to $x_j$ and $z_j$,
followed by~\eqref{eq:gauscdf2pdf} to $y_j$
\begin{align*}
&\e\biggl(
\Phi\biggl(\frac{\gamma_jy_j + \theta_jx_j-t}{\theta_\ell}\biggr) 
\Phi\biggl(\frac{\gamma_jy_j + \theta_jz_j-t}{\theta_\ell}\biggr) 
\biggr)\\
%&=
%\e\biggl(\Phi\biggl(\frac{(\gamma_jy_j-t)/\theta_\ell}
%{(1+\theta_j^2/\theta_\ell^2)^{1/2}}\biggr)^2 \biggr)\\
&=
\e\biggl(
\Phi\biggl(\frac{\gamma_jy_j-t}
{(\theta_j^2+\theta_\ell^2)^{1/2}}
\biggr)^2 
\biggr)\\
%&= \Phi\biggl(\frac{-t/(\theta_j^2+\theta_\ell^2)^{1/2}}{
%(1/(\theta_j^2+\theta_\ell^2))^{1/2}}\biggr) 
%-2T\biggl(\frac{-t/(\theta_j^2+\theta_\ell^2)^{1/2}}{
%(1/(\theta_j^2+\theta_\ell^2))^{1/2}},
%\frac{(\theta_j^2+\theta_\ell^2)^{1/2}}{(2-\theta_j^2-\theta_\ell^2)^{1/2}}\biggr)\\
&= \Phi\bigl(-t\bigr)
-2T\biggl(-t,
\frac{(\theta_j^2+\theta_\ell^2)^{1/2}}{(2-\theta_j^2-\theta_\ell^2)^{1/2}}
\biggr).
\end{align*}
Recalling  $a_1=\theta_\ell/(2-\theta_\ell^2)^{1/2}$
and $a_2=a_2(j) = (\theta_j^2+\theta_\ell^2)^{1/2}/(2-\theta_j^2-\theta_\ell^2)^{1/2}$, so
\begin{align*}
\olt^2_j &=
2T\biggl(-t,
\frac{(\theta_j^2+\theta_\ell^2)^{1/2}}{(2-\theta_j^2-\theta_\ell^2)^{1/2}}
\biggr)
-2T\biggl(-t,\frac{\theta_\ell}{(2-\theta_\ell^2)^{1/2}}\biggr)
%&=2\varphi(t)\Biggl(
%\int_0^{a_2}\frac{\varphi(tx)}{1+x^2}\rd x 
%-\int_0^{a_1}\frac{\varphi(tx)}{1+x^2}\rd x 
%\Biggr) \\
=2\varphi(t)
\int_{a_1}^{a_2}\frac{\varphi(tx)}{1+x^2}\rd x. 
\end{align*}
%The integrand is monotone decreasing in $x$
%and so the integral lies between $(a_2-a_1)\varphi(ta_1)/(1+a_1^2)$
%and $(a_2-a_1)\varphi(ta_2)/(1+a_2^2)$.

%Let's investigate $t=0$. We get
%$\olt^2_j = 2\varphi(0)(a_2-a_1)$.
The variance of $\bar f_\ell$ is
\begin{align*}
\sigma^2 & = 
\Phi(-t) -2T\biggl(-t,\frac{\theta_\ell}{(2-\theta_\ell^2)^{1/2}}\biggr)
-\Phi(-t)^2\\
%& = \Phi(t)\Phi(-t) -2T\biggl(-t,\frac{\theta_\ell}{(2-\theta_\ell^2)^{1/2}}\biggr)\\
& = \Phi(t)\Phi(-t) -2\varphi(t)\int_0^{a_1}\frac{\varphi(tx)}{1+x^2}\rd x,
\end{align*}
and so
$$
\nu(\bar f_\ell) = 
\frac{2\varphi(t) 
\sum_{j\ne\ell}\int_{a_1}^{a_2(j)}\frac{\varphi(tx)}{1+x^2}\rd x 
}
{
\Phi(t)\Phi(-t) -2\varphi(t)\int_0^{a_1}\frac{\varphi(tx)}{1+x^2}\rd x 
}
$$
establishing~\eqref{eq:preintfirst}.
For $t=0$,
\begin{align*}
\nu(\bar f_\ell) 
%& = \frac{ 2\sum_{j\ne\ell}\varphi(0)\int_{a_1}^{a_2(j)}\varphi(0)/(1+x^2)\rd x}
%{1/4 -2\varphi(0)\int_0^{a_1}\frac{\varphi(0)}{1+x^2}\rd x}\\
& = \frac{ \pi^{-1}\sum_{j\ne\ell}\int_{a_1}^{a_2(j)}(1+x^2)^{-1}\rd x}
{1/4 -\pi^{-1}\int_0^{a_1}(1+x^2)^{-1}\rd x}
%& = \frac{ \pi^{-1}\sum_{j\ne\ell}
%\bigl(\tan^{-1}(a_2(j))-\tan^{-1}(a_1)\bigr)}
%{1/4 -\tan^{-1}(a_1)/\pi}\\
 = \frac{ \sum_{j\ne\ell}
\bigl(\tan^{-1}(a_2(j))-\tan^{-1}(a_1)\bigr)}
{\pi/4 -\tan^{-1}(a_1)},
\end{align*}
establishing~\eqref{eq:preintsecond}.
Finally, if $\theta_\ell =\theta_j=1/\sqrt{d}$,
then
$ a_1 = (2d-1)^{-1/2}$ and  $a_2 = (d-1)^{-1/2}$
%$$a_1 = d^{-1/2}/(2-1/d)^{1/2} = (2d-1)^{-1/2}$$
%and 
%$$a_2 = (2d^{-1})^{1/2}/(2-2/d)^{1/2}
%=2^{1/2}/(2d-2)^{1/2} = (d-1)^{-1/2}
%$$ 
and so
\begin{align*}
\nu(\bar f_\ell) &= 
\frac{(d-1)[ \tan^{-1}( (d-1)^{-1/2})  
-\tan^{-1}( (2d-1)^{-1/2})]}
{\pi/4 -\tan^{-1}( (d-1)^{-1/2})}\\
& =
\frac{(d-1)[ d^{-1/2}+O(d^{-3/2})  
-(2d)^{-1/2}+O(d^{-3/2})  
]
}
{\pi/4 -O(d^{-1/2})},  
\end{align*} 
establishing~\eqref{eq:preintthird}.
\end{proof}

%\section*{Scratch pad}
\end{document}